\documentclass[10pt]{amsart}

\usepackage{amssymb, amsmath, amsfonts}
\usepackage{amscd}
\usepackage[all]{xy}
\numberwithin{equation}{section}

\newcommand{\pp}{\mathbb P}
\newcommand{\cc}{\mathbb C}
\newcommand{\bt}{\mathbb T}

\newcommand{\ox}{\mathcal{O}_X}
\newcommand{\Kx}{K_X}
\newcommand{\Z}{\mathbb Z}

\newcommand{\U}{\mathcal U}

\newcommand{\rank}{\mathrm{rk}\,}

\newcommand{\Hom}{\mathrm{Hom}}
\newcommand{\Ker}{\mathrm{Ker}}
\newcommand{\Sec}{\mathrm{Sec}}
\newcommand{\Pic}{\mathrm{Pic}}
\newcommand{\Gr}{\mathrm{Gr}}
\newcommand{\OG}{\mathrm{OG}}

\newtheorem{theorem}{{\textbf Theorem}}[section]
\newtheorem{prop}[theorem]{{\textbf Proposition}}
\newtheorem{cor}[theorem]{{\textbf Corollary}}
\newtheorem{lemma}[theorem]{{\textbf Lemma}}

\newtheorem{rmk}[theorem]{{\textbf Remark}}
\newenvironment{remark}{\begin{rmk}\rm}{\end{rmk}}

\title[Non-defectivity of Grassmannian bundles]{Non-defectivity of Grassmannian bundles\\over a curve}
\author{Insong Choe and George H. Hitching}

\begin{document}
\thispagestyle{empty}

\begin{abstract} Let $\Gr(2, E)$ be the Grassmann bundle of two-planes associated to a general bundle $E$ over a curve $X$. We prove that an embedding of $\Gr(2, E)$ by a certain twist of the relative Pl\"ucker map is not secant defective. This yields a new and more geometric proof of the Hirschowitz-type bound on the Lagrangian Segre invariant for orthogonal bundles over $X$, analogous to those given for vector bundles and symplectic bundles in \cite{CH1, CH2}. From the non-defectivity we also deduce an interesting feature of a general orthogonal bundle over $X$, contrasting with the classical and symplectic cases: Any maximal Lagrangian subbundle intersects at least one other maximal Lagrangian subbundle in positive rank.
\end{abstract}

\maketitle

\section{Introduction}

Let $X$ be a smooth complex projective curve of genus $g \ge 2$. 
 In the 1980s, Hirschowitz \cite{Hir} found that there do not exist vector bundles of a fixed rank and degree over $X$ with maximal subbundles of arbitrarily small degree. Precisely; let $V \to X$ be any vector bundle of rank $n \ge 2$. For $1 \le r \le n-1$, the \textsl{Segre invariant} $s_r(V)$ is defined by
\[ s_r (V) \ := \ \min \{ r \cdot \deg V - n \cdot \deg E : \hbox{ $E$ a rank $r$ subbundle of $V$} \}. \]
Hirschowitz \cite[Th\'eor\`eme 4.4]{Hir} showed that one always has $s_r (V) \le r(n-r)(g-1) + \delta$ for a certain $\delta \in \{0 , \ldots , n-1 \}$, with equality if $V$ is general. A geometric proof of this result was given in \cite[\S 5]{CH1}, exploiting the secant non-defectivity of a certain embedded Segre fibration proven in \cite[Theorem 5.1]{CH1}. This proof can be regarded as a generalization of Lange and Narasimhan's proof \cite[\S 3]{LN} of Nagata's bound for rank two bundles, which exploited the non-defectivity of certain curves in projective space.

Suppose now that $V$ admits an orthogonal or symplectic structure. We recall that a subbundle $E \subset V$ is called \textsl{Lagrangian} if $E$ is isotropic and has the largest possible rank $\left\lfloor \frac{1}{2} \rank V \right \rfloor$. The \textsl{Lagrangian Segre invariant} is defined as
\[ t(V) \ := \ \min\{ -2 \cdot \deg E : \hbox{ $E \subset V$ a Lagrangian subbundle} \}. \]
A Lagrangian subbundle will be called \textsl{maximal} if it has maximal degree among all Lagrangian subbundles. In \cite[\S 3]{CH2}, with arguments analogous to those in \cite{CH1}, the non-defectivity of certain embedded Veronese fibrations was proven and used to compute the sharp upper bound on $t(V)$ in the symplectic case. (Note that if $V$ is symplectic then $\rank V$ is even.)

In \cite{CH3}, a sharp upper bound on $t(V)$ was given for orthogonal bundles of rank $2n$, by a different method (for comparison, this is briefly sketched in Remark \ref{quotarg}). However, compared with the treatment of vector bundles and symplectic bundles in \cite{CH1, CH2}, there is a missing geometric picture in the orthogonal case, namely, the non-defectivity of the object corresponding to the aforementioned Segre and Veronese fibrations. By \cite[\S 2]{CH3}, this turns out to be an embedding of the Grassmannian bundle $\Gr (2, E)$ whose fiber at a point $x$ is the Grassmannian of planes $\Gr(2, E_x)$ for a generic vector bundle $E \to X$ of rank $n$.

The first goal of the present note is to complete the picture for orthogonal bundles by showing the non-defectivity of these Grassmannian bundles. It is relevant to point out that the Grassmannian parameterizing projective lines in $\pp^N$ is secant defective in most cases; see Catalisano--Geramita--Gimigliano \cite{CGG}.

Here is an overview of the paper. In \S 2 we recall some results on the geometry of orthogonal extensions. In \S 3 we prove the desired non-defectivity statement (Theorem \ref{nondefectivity}) for $\Gr(2, E)$. As in the classical and symplectic cases, the strategy is to describe the embedded tangent spaces of $\Gr(2, E)$ and apply Terracini's Lemma.

In \S 4, we use Theorem \ref{nondefectivity} to give a proof of the Hirschowitz-type bound on the Lagrangian Segre invariant of orthogonal bundles, analogous to those mentioned above in \cite[\S 5]{CH1} and \cite[Theorem 1.4]{CH2}.

Furthermore, in \S 5 we use Theorem \ref{nondefectivity} to answer a question which we were unable to solve with the methods in \cite{CH3}: We show that any maximal Lagrangian subbundle of a general orthogonal bundle meets another maximal Lagrangian subbundle in a sheaf of positive rank. In this way orthogonal bundles behave differently from general vector bundles and symplectic bundles. More information and precise statements are given in Theorem \ref{intersection}.

Regarding future investigations: If $Q \to X$ is a principal $G$-bundle, the notion of a subbundle or isotropic subbundle generalizes to that of a reduction of structure group to a maximal parabolic subgroup $P \subset G$; equivalently, a section $\sigma \colon X \to Q/P$. The Segre invariant $s_r(V)$ or $t(V)$ is replaced by the number
\[ s_P(Q) \ := \ \min \left\{ \deg \sigma^* T^\mathrm{vert}_{(Q/P)/X} : \hbox{ $\sigma$ a reduction of structure group to $P$} \right\} \]
where $T^\mathrm{vert}_{(Q/P)/X}$ is the tangent bundle along fibers of $Q/P \to X$. Holla and Narasimhan \cite{HN} computed an upper bound on $s_P$, which is not always sharp. The strategy of exploiting secant non-defectivity has given sharp upper bounds on certain $s_P$ if $G$ is $\mathrm{GL}_r \cc$, $\mathrm{Sp}_{2n} \cc$ or $\mathrm{SO}_{2n} \cc$. It would be interesting to investigate whether these ideas can be used to give a sharp upper bound on $s_P$ in general.

Although the present note can be read independently of \cite{CH1, CH2, CH3, CH4}, we use several results from these articles. In particular, access to \cite[\S 2 and \S 5]{CH3} may be helpful for the reader.

\section{Grassmannian bundles inside the extension spaces}

Here we recall some notions from \cite{CH2, CH3}. Let $X$ be a projective curve over $\cc$ which is smooth and irreducible of genus $g \ge 2$. Let $W$ be a vector bundle over $X$. Via Serre duality and the projection formula, there are identifications
\[ H^1 (X, W) \ \cong \ H^0 (X, \Kx \otimes W^*)^* \ \cong 
\ H^0 \left( \pp W, \pi^* \Kx \otimes {\mathcal O}_{\pp W}(1) \right)^*. \]
Thus we obtain naturally a rational map $\phi \colon \pp W \dashrightarrow \pp H^1 (X, W)$.

Suppose now that $W = \wedge^2 E$ for a vector bundle $E$ of rank $n \ge 2$. Consider the fiber bundle $\Gr(2, E)$ over $X$ whose fiber at $x \in X$ is the Grassmannian of 2-dimensional subspaces of $E_x$. Then we get a rational map
\[
\psi \colon \Gr(2,E) \dashrightarrow \pp H^{1}(X, \wedge^{2} E)
\]
by composing $\phi$ with the fiberwise Pl\"ucker embedding. In fact there is a diagram
\begin{equation}
\xymatrix{    & \pp(E \otimes E) \ \ar@{-->}^-{\widetilde{\phi}}[r] & \ \pp H^1(X, E \otimes E)   \\
   \Gr(2, E) \ \ar@{^{(}->}[r]  & \ \pp (\wedge^2 E) \ \ar@{-->}[r]^-\phi \ar@{^{(}->}[u] & \ \pp H^1(X, \wedge^2 E) . \ \ar@{^{(}->}[u]  } \label{extensions}
 \end{equation}
By the above discussion, it is easy to see that the line bundle on $\Gr(2, E)$ inducing $\psi$ is $\pi^* \Kx \otimes \det\, \U^*$, where $\U^*$ is the relative universal bundle on $\Gr(2, E)$.

Recall that the \textsl{slope} of a bundle $E$ is defined as the ratio $\mu(E) := \deg (E) / {\rm rk} (E)$. The following is a consequence of \cite[Lemma 2.2]{CH3}:
\begin{lemma} \label{embed} Let $E \to X$ be a stable bundle with $\mu(E) < -1$. Then
\[ \widetilde{\phi} \colon \pp (E \otimes E) \dashrightarrow \pp H^{1}(X, E \otimes E) \quad \hbox{and} \quad \psi \colon \Gr(2, E) \dashrightarrow \pp H^{1}(X, \wedge^2 E) \]
are embeddings. \qed \end{lemma}

\noindent Now the space $H^1(X, E \otimes E) $ is a parameter space for extensions 
\[
0 \to E \to V \to E^* \to 0.
\]
By \cite[Criterion 2.1]{Hit1}, the subspace $H^{1}(X, \wedge^2 E)$ parameterizes extensions $V$ with an orthogonal structure with respect to which $E$ is Lagrangian. As discussed in \cite{CH3}, there is a relationship between the Segre stratification on the moduli space of orthogonal bundles and the stratification given by the higher secant variety of $\Gr(2, E)$ inside $\pp H^{1}(X, \wedge^2 E)$. This motivates the work in the next section, and will be discussed in more detail in \S 4.

\section{Non-defectivity of Grassmannian bundles}

In this section, we assume that $E$ is a general stable bundle of rank $n$ and slope $\mu(E) < -1$, and consider the embedding $\psi \colon \Gr(2, E) \hookrightarrow \pp H^1 (X, \wedge^2 E)$.
For each positive integer $k$, the $k$-th secant variety $\Sec^k \Gr(2, E)$ is the Zariski closure of the union of all the linear spans of $k$ general points of $\Gr(2,E)$. 
We say that $\Gr(2, E)$ is \textit{non-defective} if for all $k \ge 1$, we have
\[
\dim \Sec^k \Gr(2, E) \ = \ \min \{ k \cdot \dim \Gr(2, E) + (k-1), \ \dim \pp H^{1}(X, \wedge^2 E) \}.
\]

\begin{theorem} \label{nondefectivity} For a general stable bundle $E$ of rank $n$ and degree $d < -n$, the Grassmannian bundle $\Gr(2, E) \subset \pp H^1 (X, \wedge^2 E)$ is non-defective.
\end{theorem}
We will prove this theorem by applying Terracini's Lemma. To do this, we must first describe the embedded tangent spaces of $\Gr(2, E)$. Now a point of $\Gr(2, E)$ corresponds to a two-dimensional subspace $P \subseteq E|_x$ for some $x \in X$. Let $\hat{E}$ be the elementary transformation of $E$ along this subspace:
\[
0 \to E \to \hat{E} \to \cc_x \oplus \cc_x \to 0.
\]
We may regard $\hat{E}$ as the sheaf of sections of $E$ which are regular apart from at most simple poles at $x$ in directions corresponding to $P$. This induces a sequence $0 \to \wedge^2 E \to \wedge^2 \hat{E} \to \tau_P \to 0$, where $\tau_P$ is torsion of degree $2(n-1)$. 
 The associated cohomology sequence is
\begin{equation}
\cdots \to \Gamma \left( \tau_P \right) \to H^1(X, \wedge^2 E) \longrightarrow H^1(X, \wedge^2 \hat{E})  \longrightarrow 0. \label{tauP}
\end{equation}
Let us describe $\tau_P$ more explicitly. Choose a local coordinate $z$ centered at $x$, and a local frame $e_1, e_2, \ldots , e_n$ of $E$ near $x$, where $e_1(x)$ and $e_2(x)$ span $P$. Then $\Gamma(\tau_P)$ has a basis consisting of the following principal parts:
\begin{equation} \frac{e_{1} \wedge e_i}{z} : \ i = 2, \ldots , n, \quad \frac{e_{j} \wedge e_2}{z}: \ j = 3, \ldots , n, \quad \hbox{and} \quad \frac{e_{1} \wedge e_2}{z^2}. \label{GammatauPbasis} \end{equation}
Note that $\frac{e_i \wedge e_j}{z}$ depends only on the values of the sections $e_i$ and $e_j$ at $x$, but $\frac{e_{1} \wedge e_2}{z^2}$ also depends on the $1$-jets of $e_1$ and $e_2$. However, the image of $\Gamma(\tau_P)$ in $H^1 (X, \wedge^2 E)$ depends only on the subspace $P$.

\begin{lemma} \label{tangent} For $P \in \Gr(2, E)$, the embedded tangent space $\bt_P \Gr(2, E)$ to $\Gr(2, E)$ at $P$ coincides with
\[ \pp \Ker \left[ H^1(X, \wedge^2 E) \longrightarrow H^1 \left( X, \wedge^2 \hat{E} \right) \right]. \]
\end{lemma}
\begin{proof}
Let $z$ and $e_1, e_2, \ldots , e_n$ be as above. For $1 \le i \le 2$, let $E_i$ be the elementary transformation of $E$ satisfying
\[ \Ker \left( E|_x \to E_i |_x \right) = \cc \cdot e_i (x). \]
Recall that the \textsl{decomposable locus} $\Delta$ of $\pp (E \otimes E)$ is defined by
\[ \Delta \ = \ \bigcup_{x \in X} \pp \{ e \otimes f : \hbox{$e, f$ nonzero in $E|_x$} \} \ \cong \ \pp E \times_X \pp E. \]
Now by hypothesis and by Lemma \ref{embed}, we also have an embedding
\[ \tilde{\phi} \colon \pp(E \otimes E) \hookrightarrow \pp H^1 (X, E \otimes E). \]
By \cite[Lemma 5.3]{CH1}, the embedded tangent space $\bt_{e_1 \otimes e_2} \Delta$ is given by
\[ \pp \Ker \left[ H^{1}(X, E \otimes E) \longrightarrow H^{1}(X, E_1 \otimes E_2) \right]. \]
Therefore, by the cohomology sequence of
\[ 0 \to E \otimes E \to E_1 \otimes E_2 \to \frac{E_1 \otimes E_2}{E \otimes E} \to 0, \]
we see that $\bt_{e_1 \otimes e_2} \Delta$ is (freely) spanned by the cohomology classes of the principal parts
\begin{equation} \frac{e_1 \otimes e_i}{z}: i = 1, \ldots , n, \quad \frac{e_j \otimes e_2}{z}: j = 2, \ldots , n, \quad \hbox{and} \quad \frac{e_1 \otimes e_2}{z^2}. \label{spanningset} \end{equation}
Now $\Gr(2, E)$ is precisely the image of $\Delta$ under the projection $E \otimes E \to \wedge^2 E$. 
 Thus the embedded tangent space to $\Gr(2, E)$ at $P$ is exactly the image of $\bt_{e_1 \otimes e_2} \Delta$ under the projection $\pp H^{1}(X, E \otimes E) \dashrightarrow \pp H^{1}(X, \wedge^{2}E)$. Hence $\bt_P \Gr(2, E)$ is spanned by the cohomology classes of the antisymmetrizations of (\ref{spanningset}):
\[ \frac{e_{1} \wedge e_i}{z}, \ i = 2, \ldots , n, \quad \frac{e_{j} \wedge e_2}{z}, \ j = 3, \ldots , n, \quad \hbox{and} \quad \frac{e_{1} \wedge e_2}{z^2}. \]
But this is exactly the basis (\ref{GammatauPbasis}). 
 The lemma follows by (\ref{tauP}). \end{proof}

Recall now that Hirschowitz' lemma \cite[\S 4.6]{Hir} states that the tensor product of two general bundles is non-special. We require also the following variant:
\begin{lemma} \label{non-special}
Suppose $F \to X$ is a general stable bundle of rank $n$ and degree $e$. Then $\wedge^2 F$ is non-special; that is,
\[
\dim H^0(X, \wedge^2 F) \ = \ \begin{cases} (n-1)e - \frac{1}{2} n(n-1)(g-1) & \text{\ if \ } e > \frac{1}{2}n(g-1),\\  0  & \text{\ if \ } e \le \frac{1}{2}n(g-1).
\end{cases}
\]
\end{lemma}
\begin{proof}
If $e \le \frac{1}{2}n(g-1)$, then by \cite[Lemma A.1]{CH2} we have $\dim H^0 (X, F \otimes F) = 0$, and hence also $\dim H^0 (X, \wedge^2 F) = 0$. The other case follows by an argument practically identical to that in \cite[Corollary A.3]{CH2}.
\end{proof}

\begin{proof}[Proof of Theorem \ref{nondefectivity}] To ease notation, write $G = \Gr(2, E)$. By the Terracini lemma, the dimension of the higher secant variety $\Sec^k G$ coincides with that of the linear span of $k$ general embedded tangent spaces:
\[
\dim (\Sec^k G) \ = \ \dim \left< \bt_{P_1} G, \bt_{P_2}G, \ldots, \bt_{P_k}G \right>,
\]
where $P_1, P_2, \ldots, P_k$ are $k$ general points of $G$ supported at $x_1, x_2 , \ldots , x_k$ respectively. For $1 \le i \le k$, let $0 \to E \to F_i \to \cc^2_{x_i} \to 0$ be the elementary transformation of $E$ at the plane $P_i$. Then by Lemma \ref{tangent}, we have
\[
\bt_{P_i} G \ = \ \pp  \Ker \left[ H^1(X, \wedge^2 E) \longrightarrow H^1(X, \wedge^2 F_i) \right].
\]
Write $F$ for the elementary transformation of $E$ determined by $P_1, \ldots, P_k$. Then $F_i$ is contained in $F$ for each $i$, and the linear span $\left< \bt_{P_i} G : 1 \le i \le k \right>$ is given by 
\[
\pp \Ker \left[ H^1(X, \wedge^2 E) \longrightarrow H^1(X, \wedge^2 F) \right].
\]
Thus, to prove the theorem, we must show that
\[
\dim \Ker \left[ H^1(X, \wedge^2 E) \longrightarrow H^1(X, \wedge^2 F) \right] \ = \ \min \{ k \cdot \dim G + k, \ \dim H^1(X, \wedge^2 E) \}.
\]
Note that
\begin{equation}
k \cdot \dim G + k \ = \ k(2n-3) + k \ = \ 2k(n-1) \ = \ \deg (\wedge^2 F) - \deg (\wedge^2 E). \label{dimSeckG}
\end{equation}

Firstly, assume that $k \cdot \dim G + k  \ < \ \dim H^1(X, \wedge^2 E)$, which is equivalent to
\[ \deg F \ = \ d + 2k \ < \ \frac{n(g-1)}{2}. \]

\noindent \textbf{Claim:} For general $E$ and general $P_1 , \ldots, P_k$ in $\Gr(2, E)$, the bundle $\wedge^2 F$ is non-special in the sense of Lemma \ref{non-special}. 

By the claim, $h^0(X, \wedge^2 F) = 0$, and so
\[
\begin{split}
\dim \Ker \left[ H^1(X, \wedge^2 E) \longrightarrow H^1(X, \wedge^2 F) \right] \ &= \ \dim H^1(X, \wedge^2 F) - \dim H^1(X, \wedge^2 E) \\
&= \ \deg (\wedge^2 F) - \deg (\wedge^2 E) \\
&= \ k \cdot \dim G + k \hbox{ by (\ref{dimSeckG})}.
\end{split}
\]

On the other hand, suppose $k \cdot \dim G + k \ge \dim H^1(X, \wedge^2 E)$, so
\[ \deg F \ \ge \ \frac{1}{2}n(g-1). \]
 By the above claim and by Lemma \ref{non-special}, then,
\[
\dim H^0 (X, \wedge^2 F) \ =  \ (n-1)(d+2k) - \frac{1}{2} n(n-1)(g-1).
\]
Therefore,
\[
\begin{split}
\dim \Ker &\left[ H^1(X, \wedge^2 E) \longrightarrow  H^1(X, \wedge^2 F) \right] \\
&= \ \dim H^1(X, \wedge^2 F) - \dim H^1(X, \wedge^2 E) - \dim H^0(X, \wedge^2 F)\\
&= \ \deg (\wedge^2 F) - \deg (\wedge^2 E) - (n-1)(d+2k) + \frac{1}{2} n(n-1)(g-1)\\
&= \ -(n-1)d + \frac{1}{2} n(n-1)(g-1) \\
&= \ \dim H^1(X, \wedge^2 E).
\end{split}
\]
Thus we are done once we have proven the claim. Note that the condition in Lemma \ref{non-special} can be restated as
\begin{equation} h^0( X, \wedge^2 F) \cdot h^1 (X, \wedge^2F) = 0. \label{nonspeccond} \end{equation}
It suffices to show that there exists a stable $E$ such that some elementary transformation $F$ of $E$ of the stated form satisfies (\ref{nonspeccond}). By Lemma \ref{non-special}, we may choose an $F_0$ satisfying (\ref{nonspeccond}). Let $E_0$ be some elementary transformation of $F_0$ fitting into a sequence
\begin{equation} 0 \to E_0 \to F_0 \to \bigoplus_{i=1}^k \cc^2_{x_i} \to 0. \label{specialF} \end{equation}
Since (\ref{nonspeccond}) is an open condition on families, it holds for a general deformation
\[ 0 \to E_t \to F_t \to \tau_{t} \to 0 \]
over a small disk $T$, where $\tau_t = \bigoplus_{i=1}^k \cc^2_{x_i(t)}$ for $t \in T$. Thus for general $t \in T$, the deformation $F_t$ satisfies (\ref{nonspeccond}). Since a general deformation of $E_0$ is a general stable bundle, we are done.
\end{proof}

Theorem \ref{nondefectivity} asserts the non-defectivity of the image of $\Gr(2, E)$ in the projective space $| \pi^* \Kx \otimes \det \, \U^*|^*$, where $\U$ is the relative universal bundle over $\Gr(2, E)$. 
 It is not difficult to generalize this to certain other line bundles over $\Gr (2, E)$ restricting to $\det \, \U^*$ on each fiber:

\begin{cor} Let $E \to X$ be a general stable bundle, and suppose $L \to X$ is a general line bundle satisfying $\deg L > 1 + \mu(E)$. Then the map
\[ \Gr(2, E) \ \dashrightarrow \ |\pi^* (\Kx L^2) \otimes \det \, \U^*|^* \]
is an embedding, and the image is secant non-defective. \end{cor}
\begin{proof} Write $\pi_1$ for the projection $\pp(\wedge^2 (E \otimes L)) \to X$, and $\U_1$ for the relative universal bundle over $\Gr(2, E \otimes L)$. A straightforward calculation shows that
\[ H^0 ( \Gr(2, E) , \pi^* (\Kx L^2) \otimes \det \, \U^* ) ^* \ \cong \ H^0 ( \Gr(2, E \otimes L^{-1}) , \pi_1^* \Kx \otimes \det \, \U_1^* )^* . \]

By hypothesis, the bundle $E \otimes L^{-1}$ is general and satisfies $\mu (E \otimes L^{-1}) < -1$. Therefore, $\Gr (2, E \otimes L^{-1}) \cong \Gr (2, E) \dashrightarrow \pp H^1 (X, \wedge^2 (E \otimes L^{-1}))$ is an embedding by Lemma \ref{embed}. By Theorem \ref{nondefectivity}, the image is secant non-defective. Since $\Gr(2, E )$ is canonically isomorphic to $\Gr(2, E \otimes L^{-1})$, the corollary follows. \end{proof}

\begin{remark} \label{psinotemb} The above definitions of the secant variety $\Sec^k \Gr(2, E)$ and non-defectivity still make sense when $\psi \colon \Gr(2, E) \dashrightarrow \pp H^1 (X, \wedge^2 E)$ is only a generically finite rational map. If we assume $E$ is such that both $\psi$ and $\tilde{\phi} \colon \Delta \dashrightarrow \pp H^1 (X, E \otimes E)$ are generically finite and rational, then the proof of Theorem \ref{nondefectivity} is valid with a few minor technical modifications. \end{remark}

\section{Application to Lagrangian Segre invariants}

We return to the study of orthogonal bundles $V$ of rank $2n$. In \cite[Theorem 1.3 (1)]{CH3}, a sharp upper bound on the value of $t(V)$ was given, based on the computation of the dimensions of certain Quot schemes. In this section we use Theorem \ref{nondefectivity} together with a lifting criterion from \cite{CH3} to give a more geometric proof of this upper bound.

Recall that the second Stiefel--Whitney class $w_2(V) \in H^2 (X, \Z/2) = \Z / 2$ is the obstruction to lifting the $\mathrm{SO}_{2n} \cc$ structure on $V$ to a spin structure (see Serman \cite{S} for details). We recall another characterisation of $w_2(V)$ from \cite[Theorem 1.2 (2)]{CH3}:

\begin{theorem} \label{parity} Let $V$ be an orthogonal bundle of rank $2n$. Then $w_2 (V)$ is trivial (resp., nontrivial) if and only if all Lagrangian subbundles of $V$ have even degree (resp., odd degree). \qed \end{theorem}

\noindent The link between the situation of Theorem \ref{nondefectivity} and the invariant $t(V)$ is given by the following:
\begin{prop} \label{secantSegre} \begin{enumerate}
\renewcommand{\labelenumi}{(\arabic{enumi})}
\item The subspace $H^{1}(X, \wedge^2 E)$ of $H^1 (X, E \otimes E)$ parameterizes extensions $0 \to E \to V \to E^* \to 0$ admitting an orthogonal structure with respect to which $E$ is Lagrangian. 
\item Let $0 \to E \to V \to E^* \to 0$ be an orthogonal extension with class $[V] \in H^1 (X, \wedge^2 E)$. Then some elementary transformation $F$ of $E^*$ satisfying $\deg(E^* / F) \le 2k$ lifts to a Lagrangian subbundle of $V$ if and only if $[V] \in \Sec^k \Gr(2, E)$. In this case, $\deg E \equiv \deg F \mod 2$.
\item If $[V] \in \Sec^k \Gr(2, E)$, then $t(V) \le 2(2k + \deg E)$.
\end{enumerate} \end{prop}
\begin{proof} (1) follows from \cite[Criterion 2.1]{Hit1}. Statement (2) is \cite[Criterion 2.2 (2)]{CH3}, and (3) is immediate from (2) and the definition of $t(V)$. \end{proof}
%
Now we can derive the upper bound on $t(V)$:

\begin{theorem} \label{upperbound} Let $V$ be an orthogonal bundle of rank $2n$. If $w_2 (V)$ is trivial (resp., nontrivial), then $t(V) \le n(g-1) + \varepsilon$, where $\varepsilon \in \{0, 1, 2, 3 \}$ is such that $n(g-1) + \varepsilon \equiv 0 \mod 4$ (resp., $n(g-1) + \varepsilon \equiv 2 \mod 4$). \end{theorem}

\begin{proof} Suppose $V$ is a general orthogonal bundle of rank $2n$ with $w_2(V)$ trivial. Firstly, we show that $V$ has a Lagrangian subbundle $F$ which is general as a vector bundle. We adapt the argument for symplectic bundles in \cite[Lemma 3.2]{CH2}: Choose an open set $U \subset X$ over which $V$ is trivial. By linear algebra, we may choose a Lagrangian subbundle $\tilde{F}$ of $V|_U$.  
Since $X$ is of dimension one, we may extend $\tilde{F}$ uniquely to a Lagrangian subbundle $F \subset V$. Deforming if necessary, we may assume $F$ is general as a vector bundle. By Proposition \ref{secantSegre} (1), the bundle $V$ is represented in the extension space $H^1 (X, \wedge^2 F)$. Moreover, $\deg F$ is even by Theorem \ref{parity}. 

We now compute the smallest value of $k$ for which $\Sec^k \Gr(2, F)$ sweeps out the whole of $\pp H^1 (X, \wedge^2 F )$, and hence must contain $[V]$. By the non-defectivity of $\Sec^k \Gr(2, F)$ proven in Theorem \ref{nondefectivity}, we have
\[ \dim \left( \Sec^k \Gr (2, F) \right) \ = \ \min \left\{ k (\dim \Gr(2, F) + 1) - 1, h^1 (X, \wedge^2 F) - 1 \right\}. \]
Therefore, the number $k$ we require is the smallest integral solution to the inequality
\[ k (\dim \Gr(2, F) + 1) - 1 \ \ge \ h^1 (X, \wedge^2 F) - 1. \]
Computing, we obtain
 $2k \ \ge \ -\deg F + \frac{1}{2}n(g-1)$. Since $\deg F$ is even, we have
\[ 2k + \deg F \ = \ \frac{1}{2} \left( n(g-1) + \varepsilon \right) \]
where $\varepsilon \in \{ 0, 1, 2, 3 \}$ is such that $n(g-1) + \varepsilon \equiv 0 \mod 4$. By Proposition \ref{secantSegre} (3), we obtain $t(V) \le n(g-1) + \varepsilon$ as required. Since $V$ was chosen to be general, the bound is valid for all $V$ by semicontinuity.

The case where $w_2(V)$ is nontrivial is proven similarly.
%
\end{proof}

\begin{remark} \label{quotarg} The above result was proven by a different method in \cite[\S 5]{CH3}, which we outline here for comparison. Consider firstly bundles with trivial $w_2$. For each even number $e \ge 0$, one constructs a family of extension spaces of the form $\pp H^1 (X, \wedge^2 E)$ with $\deg E = -e$, admitting a classifying map to the moduli space ${\mathcal MO}_{2n}^+$ of semistable orthogonal bundles of rank $2n$ over $X$ with trivial $w_2$. The fiber over a stable $V \in {\mathcal MO}_{2n}^+$ is identified with a Quot-type scheme of degree $-e$ Lagrangian subbundles of $V$. Computing the dimension of this Quot scheme, one sees that for $\varepsilon \in \{ 0, 1, 2, 3 \}$ such that $n(g-1) + \varepsilon \equiv 0 \mod 4$, the classifying map corresponding to $e = \frac{1}{2}(n(g-1) + \varepsilon)$ dominates ${\mathcal MO}_{2n}^+$. Thus $t(V) \le n(g-1) + \varepsilon$ for a general $V$ with $w_2(V)$ trivial, and hence for all $V$ by semicontinuity. A similar method works for bundles with nontrivial $w_2$, taking $e$ to be odd instead of even.
\end {remark}

\begin{remark} In Theorem \ref{upperbound}, the secant geometry of $\Gr(2, E) \subset \pp H^1 (X, \wedge^2 E)$ is applied in the ``opposite'' sense to that in \cite[\S 5]{CH3}. In Theorem \ref{upperbound}, the density of $\Sec^k \Gr(2, F)$ in $\pp H^1 (X, \wedge^2 F)$ is used to produce a \emph{lower} bound on the degrees of maximal Lagrangian subbundles of a general orthogonal extension $0 \to E \to V \to E^* \to 0$. On the other hand, in \cite[Theorem 5.2]{CH3}, the fact that certain $\Sec^k \Gr(2, E)$ are \emph{not} dense in their respective $\pp H^1 (X, \wedge^2 E)$ is used to give \emph{upper} bounds on the degrees of maximal Lagrangian subbundles of the corresponding extensions. \end{remark}

\section{Intersection of maximal Lagrangian subbundles}

Lange and Newstead showed in \cite[Proposition 2.4]{LaNe} that if $W$ is a generic vector bundle of rank $r$ and if $k \leq r/2$, then two maximal subbundles of rank $k$ in $W$ intersect generically in rank zero. The analogous statement for maximal Lagrangian subbundles of a generic symplectic bundle was proven in \cite[Theorem 4.1 (3)]{CH3}. However, as noted in \cite[Remark 5.1]{CH3}, the corresponding approach in the orthogonal case does not exclude the possibility that two maximal Lagrangian subbundles intersect in a line bundle. 
 In this section, we use the non-defectivity statement of Theorem \ref{nondefectivity} to show that this (somewhat unexpected) situation in fact arises for a general orthogonal bundle of even rank.

Firstly, we make precise the statement of \cite[Remark 5.1]{CH3}. 

\begin{prop} Suppose $X$ has genus $g \ge 5$, and $n \ge 2$. Let $V \to X$ be a general orthogonal bundle of rank $2n$. Then the generic rank of the intersection of any two maximal Lagrangian subbundles of $V$ is at most $1$. \label{dimcount}
\end{prop}
\begin{proof} Let $E$ be a general bundle of degree $-e := -\frac{1}{2}(n(g-1) + \varepsilon)$, where $\varepsilon \in \{0, 1, 2, 3\}$ is determined as in Theorem \ref{upperbound} by $n$ and $g$ together with a choice of Stiefel--Whitney class $w_2$. We consider orthogonal extensions $0 \to E \to V \to E^* \to 0$, which by Lemma \ref{secantSegre} (1) are parameterized by $H^1 (X, \wedge^2 E)$.

Suppose $H \subset E$ is a subbundle of rank $r$ and degree $-h$, and write $q \colon E \to E/H$ for the quotient map. Then by the proof of \cite[Criterion 2.3 (2)]{CH3} the bundle $H^\perp / H$ is an orthogonal extension
\[ 0 \to \frac{E}{H} \to \frac{H^\perp}{H} \to \left( \frac{E}{H} \right)^* \to 0 \]
with class $q_* {^tq^*}[V] \in H^1 (X, \wedge^2 (E/H))$. Furthermore, $V$ admits a Lagrangian subbundle $F$ of degree $-f \ge -e$ fitting into a diagram
\[ \xymatrix{ 0 \ar[r] & E \ar[r] & V \ar[r] & E^* \ar[r] & 0 \\
 0 \ar[r] & H \ar[u] \ar[r] & F \ar[u] \ar[r] & F/H \ar[u] \ar[r] & 0 } \]
if and only if
\begin{equation} e-h \ \ge \ 0 \label{emhpos} \end{equation}
and $\left[ H^\perp / H \right] = q_* {^tq^*}[V]$ belongs to the secant variety
\[ \Sec^{e-h} \Gr (2, E/H) \ \subseteq \ \pp H^1 \left(X, \wedge^2 (E/H) \right). \]
By Step 1 of the proof of \cite[Theorem 5.1]{CH3}, the locus of those $V$ in $H^1 (X, \wedge^2 E)$ admitting some such configuration of isotropic subbundles $F$ and $H$ has dimension at most
\begin{equation} e(n - r) - (e - h)(r + 1) - \frac{1}{2}(n - r)(n + r - 1)(g - 1) + h^1 (X, \wedge^2 E). \label{bigdim} \end{equation}

Suppose firstly that $r \le n-2$, so that $H^1 (X, \wedge^2 (E/H))$ is nonzero. Since $e-h \ge 0$, the above dimension is strictly smaller than $h^1 (X, \wedge^2 E)$ if
\[ \varepsilon \ < \ (r-1)(g-1). \]
Since $0 \le \varepsilon \le 3$, this is satisfied for all $g \ge 5$ if $r \ge 2$.

If $2 \le r = n-1$, then $E/H$ is a line bundle, so $h^1 (X, \wedge^2 (E/H)) = 0$. Thus
\[ \frac{H^\perp}{H} \ \cong \ \frac{E}{H} \oplus \left( \frac{E}{H} \right)^* \]
and the inverse image $F$ of $(E/H)^*$ is a Lagrangian subbundle of $V$, of degree
\[ \deg H + \deg (E/H)^* \ = \ 
 e - 2h. \]
Since $V$ is general, $\deg F = e - 2h \le -e$, so $e-h \le 0$. In view of (\ref{emhpos}), therefore, $e = h$. We claim that a general bundle $E$ of rank $n$ and degree $-e$ has no rank $n-1$ subbundle of degree $\ge -e$. By Hirschowitz \cite[Th\'eor\`eme 4.4]{Hir}, we have
\[ (n-1) \deg E - n \deg H \ = \ e \ = \ \frac{1}{2} \left( n(g-1) + \varepsilon \right) \ \ge \ (n-1)(g-1) \]
since $E$ is general. Thus we can have $\deg H = -e$ only if $\varepsilon \ge (n-2)(g-1)$. 
 But this is excluded for $g \ge 5$ since $\varepsilon \le 3$ and $n = r+1 \ge 3$ by hypothesis. \end{proof}

Next, we will need the following result from Reid \cite[\S 1]{Reid} on the even orthogonal Grassmannian $\OG(n,2n)$ parameterizing Lagrangian subspaces of $\cc^{2n}$:

\begin{prop} \label{OG} The space $\OG(n, 2n)$ consists of two disjoint, irreducible and mutually isomorphic components. Two Lagrangian subspaces $\mathsf{F}_1$ and $\mathsf{F}_2$ belong to the same component if and only if $\dim \left( \mathsf{F}_1 \cap \mathsf{F}_2 \right) \equiv n \mod 2$. \qed \end{prop}

In the same way, if $V$ is an orthogonal bundle of rank $2n$, then there is a Lagrangian Grassmannian bundle $\OG(n, V) \subset \Gr(n, V)$, also with two connected components. 
 A Lagrangian subbundle of $V$ corresponds to a section $X \to \OG(n, V)$. We write $\rank ( E_1 \cap E_2 )$ for the dimension of the intersection of $E_1$ and $E_2$ at a general point of $X$. The subbundles $E_1$ and $E_2$ belong to the same component of $\OG(n, V)$ if and only if $\rank (E_1 \cap E_2) \equiv n \mod 2$.


Now we come to the main result of this section. Fix $w_2 \in H^2 (X, \Z / 2)$. By Theorem \ref{upperbound}, a general orthogonal bundle $V$ of rank $2n$ with $w_2(V) = w_2$ satisfies $t(V) = n(g-1) + \varepsilon$ where $0 \le \varepsilon \le 3$ is determined as above by $n$, $g$ and $w_2$.

\begin{theorem} \label{intersection} Assume that $g \ge 5$ and $n \ge 2$. Let $V$ be a general orthogonal bundle as above. Let $E$ be a maximal Lagrangian subbundle which is a general vector bundle of rank $n$ and degree $-e := -\frac{1}{2}(n(g-1) + \varepsilon)$.
\begin{enumerate}
\renewcommand{\labelenumi}{(\arabic{enumi})}
\item There exist maximal Lagrangian subbundles $\widetilde{E}$ and $\widetilde{F}$ of $V$ which intersect $E$ in rank $0$ and $1$ respectively.
\item The locally free part $H$ of $E \cap \widetilde{F}$ satisfies
\[ -e \ \le \ \deg H \ \le \ - e + \frac{\varepsilon(n-1)}{4}. \]
\item If $n(g-1)$ is divisible by $4$, then $E \cap \widetilde{F}$ is a line subbundle of degree $-e$.
\end{enumerate}
\end{theorem}
\begin{proof}
(1) By Proposition \ref{secantSegre} (1), the bundle $V$ is an extension $0 \to E \to V \to E^* \to 0$, with class $[V] \in H^1 (X, \wedge^2 E)$. By the non-defectivity statement Theorem \ref{nondefectivity}, the secant variety $\Sec^e \Gr(2,E)$ fills up the whole extension space $\pp H^1 (X, \wedge^2 E)$.
 Thus by the geometric lifting criterion Proposition \ref{secantSegre} (2), the extension $V$ contains a Lagrangian subbundle $\widetilde{E}$ lifting from $E^*$ such that $\deg \widetilde{E} \ge -e$.
 Since $E$ is maximal, in fact $\deg \widetilde{E} = -e$, so $\widetilde{E}$ is also maximal. Clearly $E \cap \widetilde{E}$ is trivial at a general point.

As for $\widetilde{F}$: Let $I$ be any rank $n-1$ subbundle of $E$, which is necessarily isotropic in $V$. Over an open set $U \subset X$ upon which $V$ is trivial, complete $I|_U$ to a Lagrangian subbundle of $V|_U$ intersecting $E|_U$ in $I|_U$. This gives a section of $\OG(n,V)$ over $U$. Since $X$ is of dimension one, we can extend this uniquely to a global section over $X$, and get a Lagrangian subbundle $F$ intersecting $E$ in rank $n-1$. Write $\deg F = -f$. Since $E$ is maximal Lagrangian, we have $f \ge e$.

Consider the extension space $H^1 (X, \wedge^2 F)$, in which $V$ is represented again. We claim that $h^0 (X, \wedge^2 F) = 0$. For; if $\gamma \colon F^* \to F$ were a nonzero map, then the composition $V \to F^* \to F \to V$ would be a nonzero nilpotent endomorphism of $V$, contradicting stability. Thus no such $\gamma$ exists, so
\[ h^0 (X, \wedge^2 F) \ \le \ h^0 (X, \Hom(F^*, F)) \ = \ 0. \]
Deforming $V$, $E$, $I$ and $F|_U$ if necessary, we may assume $F$ is a general stable bundle.
By the non-defectivity proven in Theorem \ref{nondefectivity}, the secant variety $\Sec^{\frac{1}{2}(f + e)}\Gr(2, F)$ has the expected dimension
\[ \min \left\{ (e+f)(n-1) - 1, h^1 (X, \wedge^2 F) - 1 \right\}. \]
Since $f \ge e = \frac{1}{2}(n(g-1) + \varepsilon)$ and $h^0 (X, \wedge^2 F) = 0$, one checks easily that this dimension is $h^1 ( X, \wedge^2 F) - 1$. 
 Therefore $\Sec^{\frac{1}{2}(f+e)} \Gr(2,F)$ fills the extension space $\pp H^1 (X, \wedge^2 F)$. Hence by Proposition \ref{secantSegre} (2), the extension $V$ contains a Lagrangian subbundle $\widetilde{F}$ lifting from $F^*$, satisfying $\deg \widetilde{F} \ge -e$. Since $E$ is maximal, again $\deg \widetilde{F} = -e$, so $\widetilde{F}$ is a maximal Lagrangian subbundle.

Now $\rank(E \cap F) = n-1 \not\equiv n \mod 2$. Hence by Proposition \ref{OG} and the discussion following it, $E$ and $F$ belong to opposite components of $\OG(n, V)$. Then the fact that $\rank (F \cap \widetilde{F}) = 0$ implies that $F$ and $\widetilde{F}$ belong to the same component if $n$ is even, and to opposite components if $n$ is odd. In both cases, the rank of $E \cap \widetilde{F}$ is odd, so in particular non-zero. By Proposition \ref{dimcount} and by generality, $\rank (E \cap \widetilde{F}) = 1$.\\
\par
(2) 
 Let us calculate the minimum value of $h$ in order for a general extension $0 \to E \to V \to E^* \to 0$ to admit a Lagrangian subbundle of degree $-e$ whose intersection with $E$ contains a line bundle $H$ of degree $-h$. As before, the dimension of the locus of such extensions is in general bounded above by the expression (\ref{bigdim}). The required inequality in $h$ is therefore
\[ e(n - r) - (e - h)(r + 1) - \frac{1}{2}(n - r)(n + r - 1)(g - 1) \ \ge \ 0 \]
where $r=1$ and $e = \frac{1}{2}(n(g-1)+\varepsilon)$. Computing, we obtain the inequality
\[ h \ \ge \ e - \frac{\varepsilon(n - 1)}{4}. \]
Moreover, by (\ref{emhpos}), we have $h \le e$. The desired inequality follows.\\
\par
(3) If $n(g-1) \equiv 0 \mod 4$, then $\varepsilon = 0$ by Theorem \ref{upperbound}. By part (2), we have $h = e$. By the proof of Proposition \ref{dimcount}, we have $[V] \in \Ker \left( q_* {^tq}^* \right)$ for some such $H$ (here we adhere to the convention that the empty set $\Sec^0 \Gr(2, E/H)$ has dimension $-1$). Thus $H^\perp / H$ splits as $E/H \oplus (E/H)^*$, and the bundle $F$ is the inverse image of $(E/H)^*$ in $H^\perp \subset V$. Thus $E \cap F$ is exactly the line bundle $H$. %
\end{proof}

\begin{remark} If $e = \frac{1}{2}n(g-1)$, we can determine the line bundle $H$ up to a point of order two in $\Pic^0 X$. Since $[V]$ belongs to $H^1(X, \wedge^2 E)$, the extension structures $0 \to E \to V \to E^* \to 0$ and $0 \to E \to V^* \to E^* \to 0$ are isomorphic, and similarly for $0 \to F \to V \to F^* \to 0$. Thus we can identify the diagrams
\[ \xymatrix{ H \ar[r] \ar[d] & F \ar[r] \ar[d] & F/H \ar[d] \\
 E \ar[r] \ar[d] & V \ar[r] \ar[d] & E^* \ar[d] \\
 E/H \ar[r] & F^* \ar[r] & \frac{E^*}{F/H} } \quad \hbox{and} \quad \xymatrix{ \left( \frac{E^*}{F/H} \right)^* \ar[r] \ar[d] & F \ar[r] \ar[d] & \left( E/H \right)^* \ar[d] \\
 E \ar[r] \ar[d] & V^* \ar[r] \ar[d] & E^* \ar[d] \\
 \left( F/H \right)^* \ar[r] & F^* \ar[r] & H^{-1} } \]
Hence $F/H \cong (E/H)^*$. Taking determinants, we obtain $H^2 \ \cong \ \det E \cdot \det F$. \qed
\end{remark}

We conclude by describing the intersection of two maximal Lagrangian subbundles of a general orthogonal bundle of odd rank. This is an easy consequence of \cite[Proposition 5.5]{CH4}:

\begin{prop} Suppose $X$ has genus $g \ge 5$, and let $V \to X$ be a general orthogonal bundle of rank $2n+1$ where $n \ge 1$. Then any pair of maximal Lagrangian subbundles of $V$ intersect trivially in a generic fiber of $V$.
\end{prop}
\begin{proof}
We may assume $n \ge 2$, the statement being obvious for $n= 1$. Let $E$ be a general bundle of rank $n$ and degree $-e = -\frac{1}{2}((n+1)(g-1)+\varepsilon)$ as above. Let $0 \to E \to F \to \ox \to 0$ be a general extension. We consider orthogonal extensions $0 \to E \to V \to F^* \to 0$ in the sense of \cite[\S 3]{CH4}, where $F = E^\perp$. By the proof of \cite[Proposition 5.5]{CH4}, for $0 \le r \le n-2$, a general such $V$ admits no maximal Lagrangian subbundle intersecting $E$ in rank $r$ if
\[ e \ = \ \frac{1}{2} \left( (n+1)(g-1) + \varepsilon \right) \le \ \frac{1}{2}(n+r+1)(g-1). \]
This is true for any $r \ge 1$ if $\varepsilon < g-1$, which holds since $\varepsilon \le 3$ and $g \ge 5$. \end{proof}

\section*{Acknowledgements}

The first author acknowledges the support of the Korea Institute for Advanced Study (KIAS) as an associate member in 2014. The second author thanks the organizing committee of ``Vector Bundles on Algebraic Curves 2014'' for financial support and for a most interesting and enjoyable conference.

\noindent \footnotesize{Department of Mathematics, Konkuk University, 1 Hwayang-dong, Gwangjin-Gu, Seoul 143-701, Korea.\\
E-mail: \texttt{ischoe@konkuk.ac.kr}\\
\\
H\o gskolen i Oslo og Akershus, Postboks 4, St. Olavs plass, 0130 Oslo, Norway.\\
E-mail: \texttt{george.hitching@hioa.no}}

\end{document}